\RequirePackage[english]{babel}
\documentclass[12pt,twoside,reqno]{amsart}

\usepackage[OT1]{fontenc}
\usepackage{type1cm}

\usepackage{enumerate}

\usepackage{amsthm}
\RequirePackage{amsmath,amsfonts,amssymb,amsthm}

\RequirePackage[dvips]{graphicx}

\usepackage{psfrag}
\usepackage{verbatim}
\usepackage[usenames]{color}
\usepackage[dvips]{graphicx} % for gfx
\usepackage{epstopdf}

\numberwithin{equation}{section}

\usepackage[active]{srcltx}  % for the inverse search

\def\N{\mathbb N}
\def\R{\mathbb R}
\def\L{\mathcal{L}}

\def\D{\mathcal{D}}
\def\H{\mathcal H}
\def\F{\mathcal{F}}

\def\ID{\mathcal{ID}}

\newtheorem{thm}{\textbf Theorem}[section]
\newtheorem{pro}[thm]{\textbf Proposition}
\newtheorem{defi}[thm]{\textbf Definition}

\newtheorem{lem}[thm]{\textbf Lemma}
\newtheorem{cor}[thm]{\textbf Corollary}

\newtheorem{que}[thm]{\textbf Question}
\newtheorem{rem}[thm]{\textbf Remark}

  \theoremstyle{remark}

\DeclareSymbolFont{bbold}{U}{bbold}{m}{n}
\DeclareSymbolFontAlphabet{\mathbbold}{bbold}

\begin{document}

\title{On thin carpets for doubling measures }

\author{Changhao Chen, Shengyou Wen}

\address{Department of Mathematical Sciences, P.O. Box 3000, 90014
University of Oulu, Finland}
\email{changhao.chen@oulu.fi}

\address{Department of Mathematics, Hubei University, Wuhan 430062, China}
\email{sywen\_65@163.com }

\date{\today}

\thanks{The first author was supported by the Vilho, Yrj{\"o}, and Kalle V{\"a}is{\"a}l{\"a} foundation.}

\begin{abstract}
We study subsets of $\R^{d}$ which are thin for doubling measures or isotropic doubling measures. We show that any subset of $\R^{d}$ with Hausdorff dimension less than or equal to $d-1$ is thin for  isotropic doubling measures. We also prove that a self-affine set that satisfies  $OSCH$ (open set condition with holes) is thin for isotropic doubling measures. For doubling measures, we prove that Bara\'nski carpets  are thin for  doubling measures. 
%In the end we extend one formal result like the following: Let $X$ be closed subset of $\R^{d}$ and $E$ be a Borel subset of $X$ with positive Lebesgue measure, then $\mu(E)>0$ for all $d$-homogeneous measures $\mu$ on $X$.

\medskip
\noindent{\bf Keywords\,\,}\ Doubling measures, Isotropic doubling measures, Bara\'nski carpets.
\end{abstract}

\maketitle

\section{Introduction}
A Borel regular measure $\mu$ on metric space $X$ is called \emph{doubling} if there is a constant $C \geq 1$ such that 
\[
0 <\mu(B(x,2 r)\leq C\mu(B(x,r))<\infty,
\]
for any $x \in X$ and
$0<r<\infty$. We call $C$ the  \emph{doubling constant} of $\mu$. Denoted by
$\mathcal{D}(X)$ all the doubling measures on $X$. A closely related concept with doubling measures is a \emph{doubling metric space}. A metric space is called doubling metric space if there exist positive integer $N$ such that any ball of radius $r$ can be covered by a collection of $N$ balls of radius $r/2$. It's easy to see that $\D(X) \neq \emptyset$ implies  $X$ is doubling. On the other hand, if the space $X$ is  doubling and complete, then $\D(X) \neq \emptyset$, for more details see \cite{krs,ls,vk,wu2}. A subset $E$ of $X$ is called \emph{thin} for doubling measures if $\mu(E) = 0 $ for every $\mu \in \D(X)$.  Being thin  for \emph{isotropic doubling measures} is defined analogously. In this paper we are going to investigate some subsets of $\mathbb{R}^d$  are  thin for doubling measures or isotropic doubling measures. First we recall a useful estimate for doubling measures, see \cite[Chapter 13]{h} \cite{wu1}.
\begin{lem} \label{lemma1}
Let $\mu \in \D(\R^d)$  and $Q_1, Q_2$ be two cubes with $Q_1 \subset Q_2$. Then
\begin{equation}
 C^{-1} \bigg( \frac{|Q_1|}{|Q_2|} \bigg)^\beta \leq \frac{\mu(Q_1)}{\mu(Q_2)}\leq 
C\bigg(\frac{|Q_1|}{|Q_2|} \bigg)^\alpha,
\end{equation}
where $|E|$ means the diameter of $E$ and $C, \alpha, \beta$ are positive constant which only depend on $\mu$. 
\end{lem}

Lemma \ref{lemma1} implies that  every subset $E$ of  $\R^d$ with Hausdorff dimension zero is thin for doubling measures (the same argument as mass distribution principle, see \cite[Chapter 4]{f}). Various examples of  thin sets for doubling measures relate to the concept of porosity,  see \cite{hww, ojala,ors, pw, wu1}. Doubling measures give zero weight to any smooth hyper-surface, see \cite[p.40]{s}. But for every $d \geq 2$, there exist rectifiable curve in $\R^{d}$ which is not  thin, see \cite{g}. In \cite{www1} the authors asked that: \emph{Is the graph of continuous function thin for doubling measures ?} This question was negatived answered in \cite{tt}. We will show that rectifiable curves  and  graphs of continuous function are  thin  for  isotropic doubling measures. The following definition is from \cite{kmw}.
\begin{defi}
A Borel measure $\mu$ on $\R^{d}$ is isotropic doubling if there
is a constant $A \geq 1$ such that
\[
A^{-1} \leq \frac{\mu(R_{1})}{\mu(R_{2})} \leq A,
\]
whenever $R_{1}$ and $R_{2}$ are congruent rectangular boxes with
nonempty intersection. 
\end{defi}
We denote by $\ID(\R^d)$ all isotropic doubling measures on $\R^d$. Isotropic doubling measures arise from the study of \emph{$\delta$-monotone mappings}. We refer to \cite{kmw} for more details about isotropic doubling measures and $\delta$-monotone mappings. In \cite{kmw} they proved that isotropic doubling measures are  absolutely continuous to $\mathcal{H}^{d-1}$ (\emph{Hausdorff measure}) and  for every $d\geq 2$, there exists an isotropic doubling measure on $\R^d$ which is singular with respect to the Lebesgue measure. The following question of \cite{kmw}
arises naturally. \emph{Is it true that every isotropic doubling measure on
$\R^{d}, d\geq 2$, is absolutely continuous with respect to the
$s$-dimensional Hausdorff measure for all $s < d$?} This question was one of the motivations of this work. We don't know the answer. However by applying a similar estimate as Lemma \ref{lemma1} to isotropic doubling measures, we get the following result.
\begin{pro}\label{thm1}
Let $E \subset \R^{d}$ with $\dim_H E\leq d-1$, then $E$ is thin for  isotropic doubling measures on $\R^{d}$.
\end{pro}

Motivated by the  above question of \cite{kmw}, we consider the \emph{self-affine sets}. By
adding the condition $OSCH$ (see Definition \ref{OSCH}) on self-affine sets, we have the following result.

\begin{thm}\label{self-affine}
A self-affine set that satisfies $OSCH$ is thin for isotropic doubling measures.
\end{thm}

For the doubling measures, things become more complicated. So we consider a special class of \emph{self-affine carpets} on the plane. Bara\'nski \cite{bar} generalized the construction of Bedford-McMullen carpets to build a class of self-affine carpets. We call them Bara\'nski carpets, see Definition \ref{bb} or \cite{bar}.  For Bedford-McMullen carpets, see \cite{b, f, mc}. For Bara\'nski carpets, we have the following result.

%construct a class of  affine  constructions generated by an arbitrary partition of the unit square by a finite number of horizontal and vertical lines. 

\begin{thm}\label{carpets}
Bara\'nski carpets are thin for doubling measures.
\end{thm}

\bigskip
\noindent\textbf{Acknowledgement:} 
We would like to thank Tuomo Ojala and Ville Suomala for reading the
manuscript and giving many helpful comments. 

\section{Isotropic doubling measures}

We start from a useful lemma of \cite{kmw}.
\begin{lem}\label{isotropic}
Let $\mu$ be a isotropic doubling measure on $\R^{d},
d\geq2$ with doubling constant $A$. Then: (i) For any congruent rectangular boxes $R_{1}, R_{2} \subset
\mathbb{R}^{d}.$
\begin{equation}
A^{-m} \leq \frac{\mu(R_{1})}{\mu(R_{2})} \leq A^{m},
\end{equation}
where $m = [\frac{dist(R_{1},R_{2})}{diam R_{1}}]+1.$

(ii) Let $Q \subset \mathbb{R}^{d}$ be a cube, and let $F$ be a face
of $Q$. The pushforward $\pi_{\sharp}\mu_{Q}$ of $\mu_{Q}$ under the
orthogonal projection $\pi : Q \rightarrow F$ is comparable to
$\mathcal{L}^{d-1}_{F}$ with constants that depend only on $d$ and $A$.
\end{lem}

The following is an analogue of Lemma \ref{lemma1} for isotropic doubling measures.
\begin{lem} \label{key lemma}
Let  $\mu \in \ID([0, 1]^{d}), d\geq 2$ with the doubling constant $A$ and $I$ be a cube in $[0,1]^{d-1}$, and $J$ be interval in $[0,1]$. Then 
\begin{equation}\label{main use}
C^{-1} |I|^{d-1} |J|^\beta\leq\mu(I\times J) \leq C |I|^{d-1}|J|^{\alpha},
\end{equation}
where $\alpha,\beta$ and  $C$ are positive constants depending only  on $d$ and $A$. 
\end{lem}
\begin{proof}
Let $\nu(E):= \mu(I\times E)$  for $E \subset [0,1]$. Then $\nu$ is a doubling measure on $[0,1]$ with the doubling constant $A$. Applying  
Lemma \ref{lemma1} to $\nu$, we have
\begin{equation}\label{newmeasure}
C_1^{-1}|J|^{\beta}\leq\frac{\nu(J)}{\nu([0,1])}\leq C_1 |J|^{\alpha},
\end{equation}
where $C_1, \alpha$, and $\beta$ are positive constants depending only on $A$. Applying the second part of  Lemma \ref{isotropic}, we see that $\nu([0,1])$ is comparable  to $|I|^{d-1}$ with the constant  depending on $d$ and $A$ only. Thus we  have finished the proof.   
\end{proof}
Applying Lemma \ref{key lemma} and the same argument as mass distribution principle (see \cite[Chapter 4]{f}), we arrive at the following  corollary immediately. 
\begin{cor}\label{holder}
Let  $\mu \in \ID([0,1]^d)$, then there exist a positive constant $\alpha$ which only depends on $\mu$, such that $\mu$ is absolutely continuous to $\H^{d-1+\alpha}$. 
\end{cor}
\begin{proof}[Proof of Theorem \ref{thm1}.]
Let $\mu \in \ID([0,1]^d)$. If $d=1$, then $\mu$ is doubling measure on $[0,1]$ and $E$ has Hausdorff dimension zero.  By Lemma \ref{lemma1}, we know that any set with Hausdorff dimension zero is thin for doubling measures. Thus we arrive at the result for $d=1$.

For the case $d \geq 2$, applying the Corollary \ref{holder}, there is positive $\alpha$ such that $\mu$ is absolutely continuous to $\H^{d-1+\alpha}$. 
Since $\dim_H E \leq d-1$, so $\H^{d-1+\alpha}(E)=0$ and thus $\mu(E)=0$. We complete the proof by the arbitrary choice of $\mu\in \ID([0,1]^d)$.
\end{proof}

Since any $k$-rectifiable sets of $\R^d$ ( \cite[Chapter 15]{ma}) have Hausdorff dimension $k$, for $k<d$ they are thin for isotropic doubling measures on $\R^d$. Let $f : [0, 1] \rightarrow \R$ be a function. Recall that the graph of function  $f$ is $G(f):=\{(x,f(x)): x \in[0,1]\}$. Now we are going to apply Lemma \ref{key lemma} to prove that the graphs of continuous functions are thin for isotropic doubling measures.

\begin{pro}\label{graph}
Let $f: [0,1]^d \rightarrow [0,1]$ be a continuous function. Then $\mu(G(f))=0$ for all $\mu \in \ID([0,1]^{d+1})$.
\end{pro}
\begin{proof}
Let $\mu \in \ID([0,1]^{d+1})$, then there is positive $C$ and $\alpha$ such that the estimate $\eqref{main use}$ holds. Since $f$ is continuous on $[0,1]^d$, it's well know that $f$ is uniformly continuous on $[0,1]^d$. Thus for any  $\epsilon >0$, there is $\delta$ such that 
$|f(x)-f(y)|\leq
\delta$ for all $x,y \in [0,1]^d$ with $|x-y|\leq \delta$.

Choose $n \in \N$, such that $2^{-n}\sqrt{d} \leq \delta$. Let $\D_{n}$ denote all the dyadic cubes of  $[0,1]^d$ with side-length $2^{-n}$. For each  cube $I$ of $\D_{n}$, there is an interval $I'\subset [0,1]$ with $|I'|\leq \epsilon$ such that 
$\{(x,f(x)): x \in I\} \subset I \times I'$. Whence
\[
G(f) \subset \cup_{I \in \D_{n}}I\times I'.  
\]
By applying Lemma \ref{key lemma}, we have
\begin{equation}
\mu(G(f))\leq \sum_{I \in \D_{n}} \mu(I\times I')\leq \sum_{I \in \D_{n}} C\epsilon^{\alpha}|I|^{d}\leq C\epsilon^{\alpha}(\sqrt{d})^{d}.
\end{equation}
Let $\epsilon\rightarrow 0$, then we have $\mu(G(f))=0$. We finish the proof by the arbitrary choice of $\mu$.   
\end{proof}

By using the same idea (applying Lusin theorem) as in \cite{tt}, the result of  Proposition \ref{graph} is also holds if we change continuous function to measurable function.
   
\begin{cor}
Let $f: [0,1]^d \rightarrow [0,1]$ be a measurable  function with respect to Lebesgue measure. Then $G(f)$ is thin for doubling measures on $[0,1]^{d+1}$.
\end{cor}
\begin{proof}
 Applying Lusin theorem (and it's normal corollary), for any $\epsilon >0$, there is a continuous function $g: [0,1]^d \rightarrow [0,1]$ such that 
\begin{equation}\label{near1}
\L^d(\{x: f(x) \neq g(x)\})< \epsilon.
\end{equation}
Let $\mu \in \ID([0,1]^d)$. Lemma \ref{isotropic}(ii) says that there is a constant $C$ which depends on $d$ and $\mu$ only, such that
\begin{equation}\label{near2}
\mu(A\times [0,1]) \leq C \L^d(A) \text{ for any } A \subset [0,1]^d.
\end{equation}
Let $D= \{x: f(x) \neq g(x)\}$. Since 
\[
G(f)= G(g) \cup (G(f) \backslash G(g)) \subset G(g) \cup (D\times [0,1]),
\]
the estimate \eqref{near2} and Proposition \ref{graph}, we have 
$\mu(G(f)) \leq C\epsilon.$ By the arbitrary choice of $\epsilon$, we have $\mu(G(f))=0$. Thus $G(f)$ is thin for isotropic doubling measures on $[0,1]^{d+1}$. 
\end{proof}

\section{Doubling measures and self-affine sets}

Let $\Lambda$ be the attractor
of the self-affine IFS
\begin{equation}
\F := \{f_i(x)=T_i x +t_i\}^m_{i=1}, x\in \R^d.
\end{equation}
We always assume that the maps $f_i$ are contractive and $T_i$ are non-singular linear maps for each $1\leq i \leq m$. For more details on self-affine sets, see \cite[Chapter 9]{f}. The following condition is often used to avoid overlap of IFS. Recall that the IFS $\F$ is said to satisfy the open set condition ($OSC$) if  
there exists a
non-empty open set $V \subset \R^d$ such that
\begin{itemize}
\item$f_i(V ) \subset V$ holds for all $1\leq i \leq m$;

\item $f_i(V ) \cap f_j(V ) = \emptyset$ for all $i \neq j$.
\end{itemize}

We recall some standard notation for  IFS. Let $S=\{1,\cdots, m\}$. Denote  $S^\ast:= \cup^{\infty}_{k=1} S^k$ all the finite words and $S^\N$ all the infinite words. Let $\sigma =(i_1, \cdots, i_k) \in S^k$. Define $f_\sigma:= f_{i_1} \circ \cdots \circ f_{i_k}$.

\begin{defi}\label{OSCH}
We say that the IFS $\F$ satisfies the $OSCH$ (open set condition with hole) if $\F$ satisfies the $OSC$ and for the open set $V$,  $V \backslash \bigcup^m_{i=1} f_i(V)$ has non empty interior.  
\end{defi}

\begin{defi}
A map $f$  is called non-singular affine map on $\R^d$, if there is a non-singular linear map $T$ and a vector $t \in \R^d$ such that $f(x)=T(x)+t, x \in  \R^d$.  A map $f$  is called diagonal affine map on $\R^d$, if there is a diagonal matrix $D=D(\lambda_1,\cdots, \lambda_d)$ and a vector $t \in \R^d$ such that $f(x)=D(x)+t, x \in  \R^d$. Note that the diagonal maps and non-singular linear  maps coincide when $d=1$.  We call  a cube $Q$ stable if  $Q \subset [0,1]^d$ and there is $x \in \R^d$ and $\rho >0$ such that $Q = x+\rho [0,1]^d$. 
\end{defi}

\begin{lem}\label{dia}
Let $Q \subset [0,1]^d$ be stable cube and $\mu \in \ID(\R^d)$.  Then there exist a positive constant $C$ (depending on $\mu$ and side-length of $Q$) such that for any diagonal affine map $f$, we have 
\begin{equation}\label{good}
\mu(f(Q)) \geq C\mu(f([0,1]^d)).
\end{equation}
\end{lem}
\begin{proof}
If $d=1$, then we arrive at the estimate by  applying the Lemma \ref{lemma1}. Now we consider the case  $d\geq 2.$ We assume that $Q=I_1 \times \cdots \times I_d$ where $I_i \subset [0,1]$. Denote by $a$ the side-length of $Q$ and $V_Q = I_1 \times [0, 1]^{d-1}$. Let $\{Q_i\}_{i=1}^{N}$ be a sequence of closed cubes with the same edge length $a$ and  disjoint interior. Furthermore we ask that $\{Q_i\}_{i=1}^{N}$ satisfies $V_Q \subset \bigcup_{i=1}^N Q_i \subset [-2,2]^d$ and for any $Q_i$ and $Q_j$, $i\neq j$, there exist $i_1, i_2,\cdots,i_n$ such that $Q_i = Q_{i_1}, Q_j=Q_{i_n}$ and $Q_{i_k} \cap Q_{i_{k+1}}\neq \emptyset$ for $1\leq k \leq n-1$. By a simple volume argument,  $N \leq (\frac{4}{a})^d$. Thus 
\begin{equation}
\mu(Q)\geq A^{-N}\mu(Q_i) ~~~\text{for all} ~~~1\leq i \leq N.
\end{equation}
Summing both sides over index $i$, we have  
\begin{equation}
\mu(Q) \geq \frac{1}{N A^N} \mu(V_Q).
\end{equation}
Let $C_1 = (\frac{a}{4})^d A^{-(\frac{4}{a})^d}$, then $\mu(Q)\geq C_1 \mu(V_Q)$.

Since $f$ is a diagonal map,  we have that $f(Q_i)$  is a rectangle for  $1\leq i\leq N$. Applying the same argument as above, we have   
\begin{equation}
\mu(f(Q))\geq A^{-N} \mu(f(Q_i))~~~ \text{for all}~~~ 1\leq i \leq N,
\end{equation}
and
$\mu(f(Q))\geq C_1 \mu(f(V_Q))$.

Applying the same argument to $V_Q$ and $[0,1]^d$ (in place of $Q, V_Q$), we have $\mu(f(V_Q)) \geq C_2 \mu(f([0,1]^d))$ where $C_2$ is a positive constant that depends on $Q$ and the doubling constant $A$ only.
Letting $C= C_1C_2$ we complete the proof.
\end{proof}

Now we are going to show that the above result also holds for any non-singular linear map. We will use the polar decomposition of a matrix. The polar decomposition says that for any matrix $T$, there exists  a symmetric matrix $S$ and orthogonal matrix $O$ such that $T=OS$. 
Furthermore if $T$ is non-singular, then $S$ is positive definite. For more details see \cite[Chapter 3]{eg}.

\begin{figure}
\centering 
\includegraphics[width=0.5\textwidth]{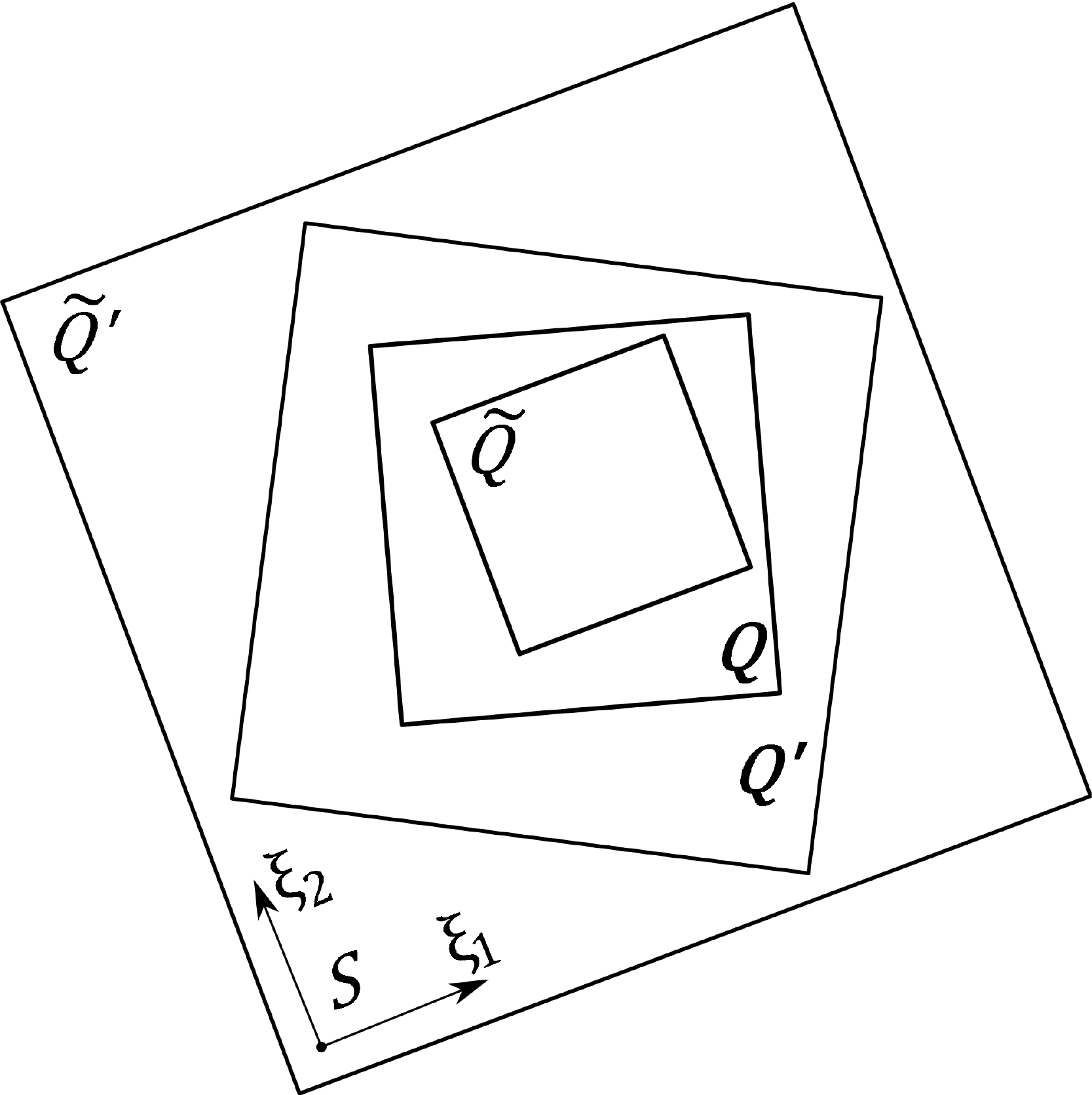}
\caption{The cubes of Proposition \ref{fu} for $d=2$.}
\label{figure1}
\end{figure}

\begin{pro}\label{fu}
Let $Q \subset Q'$ and $\mu \in \ID(\R^d)$. Then for any non-singular affine map $f$, we have 
\begin{equation}\label{need}
\mu(f(Q)) \geq C\mu(f(Q')),
\end{equation}
where $C$ is a positive constant doesn't depends on $f$. 
\end{pro}
\begin{proof}
If $d=1$, then the non-singular map $f$ is the same as diagonal map. Thus we arrive at the estimate by  applying Lemma \ref{lemma1} again.  Now we  consider $d\geq 2.$ For a non-singular map $f$, there is  a non-singular matrix $T$ and a vector $t \in \R^d,$ such that $f(x)=Tx+t$ for $x \in  \R^d$. For the convenience we use the same notations  as above writing $T=OS$. 

For positive definite matrix $S$,  it's well known that there exist a standard orthogonal basis $\{\xi_1, \cdots, \xi_d\}$ such that $S\xi_i=\lambda_i \xi_i,$ and $\lambda_i >0$ for all $1\leq i \leq d$. Let $I(\xi_i):=\{t\xi_i : t \in [-1/2,1/2]\}, 1\leq i \leq d$ and $Q_S= I(\xi_1)\times\cdots \times I(\xi_d)$ be the unite cube.
Let  $\widetilde{Q} (x,\rho):=x+\rho Q_S.$ Denote by $a$ the side-length of $Q$ and $x_0$ the center of $Q$.  By a simple geometric argument, we have 
\[
\widetilde{Q} (x_0, \frac{a}{\sqrt{d}}) \subset B(x_0, \frac{a}{2}) \subset Q.
\]
Denote by $a'$ the side-length of $Q'$ and $x_0'$ the center of $Q'$. Again by a simple geometric argument we have 
\[
Q' \subset B(x_0', \frac{a'\sqrt{d}}{2}) \subset \widetilde{Q'} (x_0', a'\sqrt{d}).
\]

Applying the same argument as in Lemma \ref{dia} to $\widetilde{Q} \subset \widetilde{Q'}$, there is a positive constant $C$ such that
\begin{equation}\label{55}
\mu(S\widetilde{Q})\geq C \mu(S\widetilde{Q'}). 
\end{equation}
Note that the  estimate \eqref{55} still holds with the same constant $C$ after rotations and translations. Thus
\begin{equation}\label{1}
\mu(f(\widetilde{Q}))\geq C\mu( f(\widetilde{Q'})).
\end{equation}
This completes the proof since $f(\widetilde{Q})\subset f(Q)$ and 
$f(Q')\subset f(\widetilde{Q'})$. 
\end{proof}

\begin{proof}[Proof of Theorem \ref{self-affine}]
Let $\Lambda$ be the attractor
of the self-affine $IFS$
\[
\{f_i(x)=T_i x +t_i\}^m_{i=1}, x\in \R^d
\]
which satisfies the $OSCH$. Since the IFS satisfies $OSCH$, there is an open set $V$ and a cube $Q$ with non empty interior such that $Q \subset V \backslash \bigcup^m_{i=1} f_i(V)$. It's well known that (see \cite[Chapter 9]{f}) there is compact cube $Q'$ such that  
$f_i(Q') \subset Q', 1\leq i \leq m, V\subset Q'$ and 
\[
\Lambda = \bigcap^\infty_{k=1}\bigcup^{\infty}_{\sigma\in S^k}f_\sigma(Q').
\]

Since our $IFS$ satisfies $OSCH$ and by the position of $Q$, we have $f_\sigma(Q) \cap f_\tau (Q)=\emptyset$ for any $\sigma\neq \tau$ where $\sigma, \tau \in S^\ast$. For each $k \in \N$, denote $G_k = \bigcup_{\sigma \in S^k} f_\sigma(Q)$. Let $\mu \in \ID(\R^d)$. Note that $Q \subset Q'$ and $f_\sigma$ is non-singular affine map for 
any $\sigma  \in S^\ast$. Thus by Proposition \ref{fu}, there exists a positive constant $C$ such that 
\begin{equation}\label{use}
\mu(f_\sigma(Q)) \geq C\mu(f_\sigma(Q')), \text{for any} ~~\sigma \in S^\ast.
\end{equation}

Since $f_{\sigma}(Q)$ are pair disjoint for $\sigma \in S^\ast$, we have 
\[
\mu(G_k)= \sum_{\sigma \in S^k} \mu(f_{\sigma}(Q)).
\]
Summing  two sides of equation \eqref{use} over $\sigma \in S^k$, we get
\begin{equation}\label{final}
\mu(G_k)\geq C \sum_{\sigma \in S^k} \mu(f_\sigma(Q'))\geq C \mu(\Lambda),
\end{equation}
where the last inequality holds since $\Lambda \subset \bigcup_{\sigma \in S^k} f_\sigma(Q')$. 

Since $\bigcup^\infty_{k=1} G_k \subset Q'$, and $G_i \cap G_j = \emptyset$ for $i\neq j$, together with inequality \eqref{final}, we have 
\begin{equation}
\infty> \sum^{\infty}_{k=1}\mu(G_k) \geq C \sum^{\infty}_{k=1}\mu(\Lambda).
\end{equation}
Thus we have $\mu(\Lambda)=0$.
\end{proof}

We don't know whether  Theorem \ref{self-affine} is also holds for doubling measures. 
\begin{que}
Is the attractor of $IFS$ satisfies $OSCH$ thin for doubling measures? 
\end{que}

Now we are going to prove Theorem \ref{carpets}. We first recall the construction of Bara\'nski carpets, see \cite{bar}.

\begin{figure}
\centering 
\includegraphics[width=0.9\textwidth]{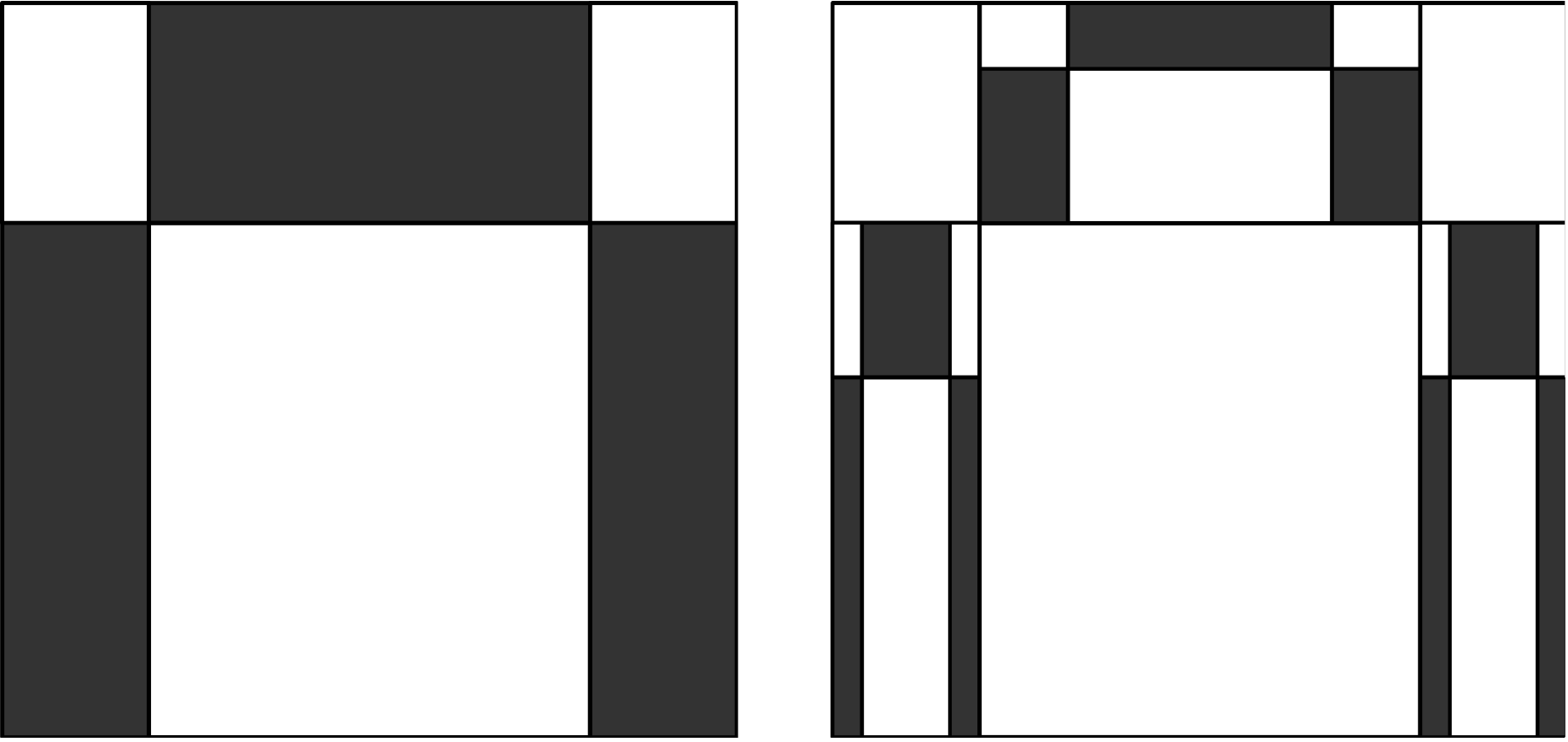}
\caption{The first two steps in the construction of $E$.}
\label{figure2}
\end{figure}

\begin{defi}\label{bb}
Let $\{a_i\}^p_{i=1}, \{b_j\}^q_{j=1}$ be two sequences of positive numbers such that $\sum^p_{i=1} a_i =1$ and $\sum^q_{j=1} b_j =1$ where $p,q \in \N$ and $p, q \geq 2$. We  have a partition of the unit square by $q$ horizontal lines and $p$ vertical lines. We exclude a sub-collection of these rectangles to form $E_{1}$ (We assume that at least  one rectangle was excluded  to avoid the trivial case). Iterate this construction for each rectangle of $E_1$ as above, in other words we  replace each rectangle of $E_1$ by an affine copy of $E_{1}$. In the end we have a limit set $E$. For an example see Figure \ref{figure2}. Recall that the limit set $E$ is called $BM$ (Bedford-McMullen) carpet if $a_i= 1/p$ for all $1\leq i\leq p$ and $b_j=1/q$ for all  $1\leq j\leq q$, see \cite[chapter 9]{f}. 
\end{defi}

Recall that a subset $E \subset \R^d$ is called porous if there exists $\alpha \in (0,1)$, such that for any ball $B(x, r)$, there is a ball $B(y, \alpha r) \subset B(x,r)$ satisfies $B(y,\alpha r)\cap E = \emptyset$. The concept of porosity is closely related to
the Assouad dimension. The connection is the following: A subset $E$ of $\R^d$ is
porous if and only if $\dim_A E < d$. For  more details, we refer to \cite[Theorem 5.2]{luu}. It's well know that if a set $E \subset \R^d$ is porous, then $E$ is thin for doubling measures on $\R^d$ (by applying the density argument for doubling measure on the porosity set, the same as \cite[p.40]{s}). Since the Assouad dimension
of Bara\'nski carpets can be less than or equal to $2$, see \cite{Fraser, Mackay}, thus we can't obtain Theorem \ref{carpets} by apply the above mentioned  result: if a set $E \subset \R^2$ has Assouad dimension less than $2$,  then $E$ is porous, and so $E$  is thin for doubling measures.

\begin{figure}
\centering 
\includegraphics[width=0.6\textwidth]{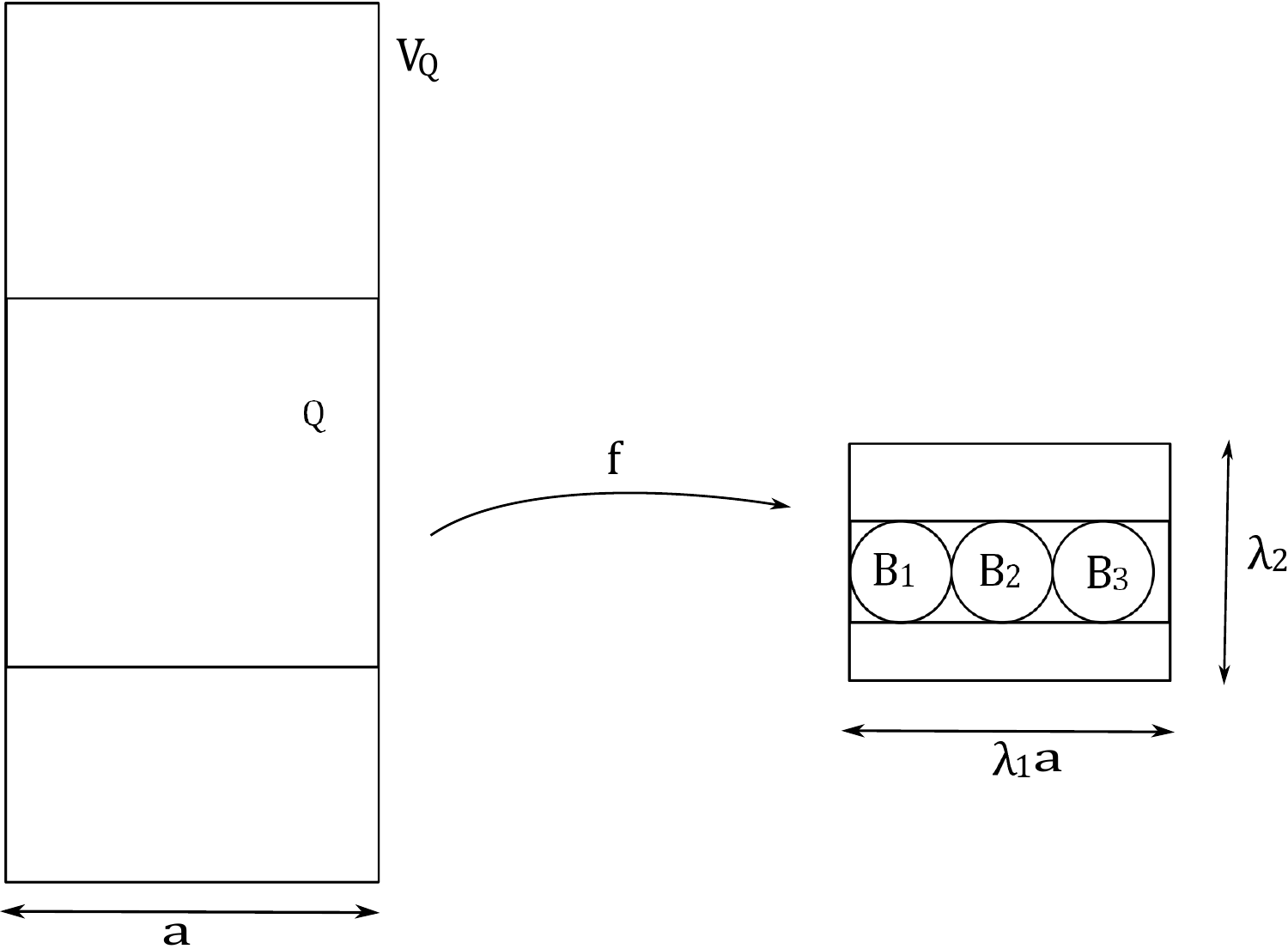}
\caption{}
\label{figure3}
\end{figure}
\begin{lem}\label{dd}
Let $Q \subset [0,1]^2$ be a cube with edge length $a$ and $\mu \in \D([0,1]^2)$. Denoted by $V_Q$ the smallest vertical strip of $[0,1]^2$ which contains $Q$.  Then there is a positive constant $C$ depending on $\mu$ and $a$ only, such that for any diagonal affine map $f(x):=D(\lambda_1, \lambda_2)(x)+t$ with $\lambda_1 \geq \lambda_2>0$, we have 
\begin{equation}\label{dgood}
\mu(f(Q)) \geq C\mu(f(V_Q)).
\end{equation}
\end{lem}
\begin{proof}
Since $f$ is diagonal map, $f(Q)$ is a rectangle with sides $a\lambda_1$ and $a\lambda_2$. We are going to place a sequence of closed balls with the same diameter $a\lambda_2$ inside the rectangle $f(Q)$. We put the first ball $B_1$ at the left part of $f(Q)$ and touching the left boundary of $f(Q)$. We put the second ball $B_2$ touching the first ball $B_1$ with disjoint interior. We continue to put the balls in the above way, see Figure \ref{figure3}. In the end we have $\lfloor\frac{\lambda_1}{\lambda_2}\rfloor$ balls inside $f(Q)$ where $\lfloor\frac{\lambda_1}{\lambda_2}\rfloor$ is the integer part of $\frac{\lambda_1}{\lambda_2}$. By a simple  geometric estimate, we have that
\[
f(V_Q) \subset \bigcup_{i=1}^{\lfloor\frac{\lambda_1}{\lambda_2}\rfloor} \frac{2}{a} B_i, 
\]
where $\rho B(x,r):=B(x,\rho r)$. Thus we have 
\begin{equation}
\mu(f(V_Q))\leq \sum^{\lfloor\frac{\lambda_1}{\lambda_2}\rfloor}_{i=1} \mu(\frac{2}{a}B_i)\leq A\sum^{{\lfloor\frac{\lambda_1}{\lambda_2}\rfloor}}_{i=1}\mu(B_i)\leq A \mu(f(Q)),
\end{equation}
the second inequality holds by the doubling property of  $\mu$ and the constant $A$ depends on $\mu$ and $a$ only. Thus we have completed the proof with  $C=1/A$.
\end{proof}

\begin{proof}[Proof of Theorem \ref{carpets}.]
Let $E$ be a Bara\'nski carpet constructed as above. For every $n \in \N$, let  $R: = I_a\times I_b$ be an $n$-level rectangle of $E_n$ where $I_a, I_b$ are subintervals of $[0,1]$ whit length $a$ and $b$ respectively. Without loss of generality we  assume $a \geq b$. For the interval $I_a$, there exists a sequence of subintervals of $I_a$ denoted by $\{I_{\sigma(i)}\}^{N(R)}_{i=1}$ where $\sigma(i)\in S^{\ast}$ and $N(R) \in \N$. These subintervals have the following properties (Moran cover),
\begin{itemize}
\item  $I_a \subset \bigcup_{i} I_{\sigma(i)}$,

\item $a_{\min}\vert I_{\sigma(i)} \vert<b\leq \vert I_{\sigma(i)}\vert, 1\leq i \leq N(R)$, where 
$a_{\min}= \min_{1 \leq i\leq p}\{a_i \}$.

\item $int(I_{\sigma(i)}) \cap int(I_{\sigma(j)}) = \emptyset$ for $i \neq j$ where $int(I)$ means the interior of $I$. 
\end{itemize}
Thus we may write $R = \bigcup^{N(R)}_{i=1} I_{\sigma(i)} \times I_b$. Let $R_i:=I_{\sigma(i)} \times I_b$ and 
$I_{\sigma(i)} \times I_b= \bigcup^{q^{\vert \sigma(i)\vert}} _{j=1}R_{i,j}$ where $R_{i,j}$ is $(n+\vert\sigma(i)\vert)$-level rectangle with $int(R_{i,j}) \cap int(R_{i,j'}) =\emptyset$ for $ j\neq j'$. There is a cube $Q \subset [0,1]^2$ such that $Q \cap E_1 = \emptyset$. 
We use the same notation $V_Q$ as in Lemma \ref{dd}. 
For each $R_{i,j}, 1\leq i \leq N(R), 1\leq j \leq q^{\vert \sigma(i)\vert}$, 
denote by $f_{i,j}$ the affine map such that $f_{i,j} ([0,1]^2)= R_{i,j}$. See Figure \ref{figure4}. Denote
\[ 
G(R):= \bigcup^{N(R)}_{i=1}\bigcup_{j=1}^{q^{\vert \sigma(i)\vert}} f_{i,j}(Q). 
\]

\begin{figure}
\centering 
\includegraphics[width=0.9\textwidth]{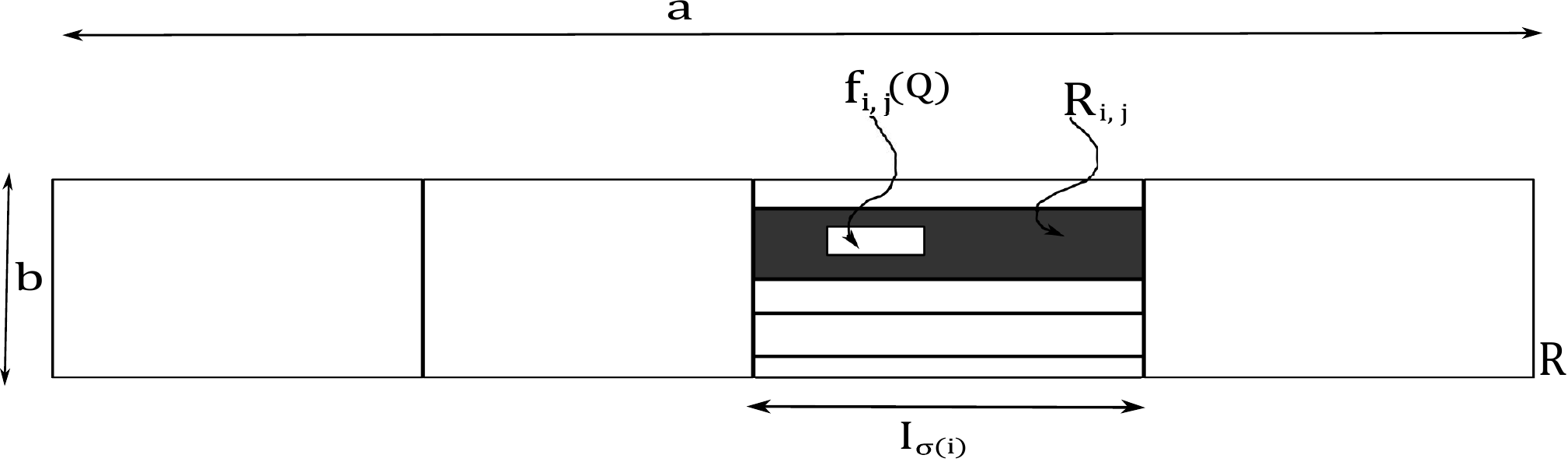}
\caption{For each $R_{i,j}$ there is a hole $f_{i,j}(Q)$.}
\label{figure4}
\end{figure}

Let $\mu \in \D([0,1]^2)$. By Lemma  \ref{dd}, there is a positive constant $C_1$ such that $\mu(f_{i,j}(Q))\geq C_1 \mu(f_{i,j}(V_Q)).$ Summing both sides over $j$ to get 
\begin{equation}\label{t1}
\sum_{j=1}^{q^{\vert\sigma(i)\vert}}\mu(f_{i,j}(Q))\geq C_1 \sum_{j=1}^{q^{\vert\sigma(i)\vert}}\mu(f_{i,j}(V_Q)). 
\end{equation}
Let $\widetilde{R_{i}}:= \bigcup^{q^{\vert\sigma(i)\vert}}_{j=1}f_{i,j}(V_Q)$. Notice that the side-length of $\widetilde{R_{i}}$ are comparable with $\vert I_{\sigma(i)}\vert$ for each $1\leq i\leq N(R)$. Thus there is a positive constant $C_2$ such that $\mu(\widetilde{R_{i}}) \geq C_2 \mu(R_i)$. Summing both sides over $i$,  we have 
\begin{equation}\label{t2}
\sum^{N(R)}_{i=1}\mu(\widetilde{R_{i}}) \geq C_2 \sum^{N(R)}_{i=1}\mu(R_i).
\end{equation}

Combine the estimates \eqref{t1} and \eqref{t2}, we arrive 
\begin{equation}\label{t3}
\mu(G(R)) \geq C_1C_2 \mu(R).
\end{equation}

Let $G_n := \cup_{R \in E_n} G(R)$,
then by estimate  \eqref{t3}, we have  that 
\begin{equation}\label{99}
\mu(G_n)\geq C_1C_2 \mu(E_n)\geq C_1C_2\mu(E).
\end{equation}

Given $n_k$, let  $\widetilde{n}_k= \max \{N(R): R \in E_n\}$ and $n_{k+1}=n_k+ \widetilde{n}_k+10$.
Let $k=1$, then we have a sequence $n_k$ and $G_{n_k}$. 
By our choice of $n_k$ and  $Q \cap E_1 = \emptyset$, we observe that the sets
 $G_{n_k}$ are  pairwise  disjoint subsets of $[0,1]^2$. Applying  estimate \eqref{99} to every $G_{n_k}$, we obtain 
\[
\infty>\mu(\bigcup_{k=1}^{\infty}G_{n_k})
\geq C_1C_2\sum^{\infty}_{k=1} \mu(E_{n_k}) \geq C_1C_2\sum^{\infty}_{k=1} \mu(E).
\]
Thus we have $\mu(E)=0$. We complete the proof by the arbitrary choice of $ \mu \in\D([0,1]^2)$.
\end{proof}

\subsection{Bedford-McMullen sponges in $\R^3$}(Suggested by V. Suomala) Applying similar argument as in the proof of Theorem \ref{carpets}, we are going to prove that \emph{Bedford-Mcmullen (BM) sponges}  are thin for doubling measures on $[0,1]^3$. We show the construction of $BM$ sponges first. Let $p,q,u \in \N$ and $2\leq p\leq q\leq u$. Divide $[0,1]^3$ into $p \times q \times u$ rectangles of sides $1/p, 1/q$ and $1/u$. Select a subcollection of
these rectangles to form $E_1$. Iterate this construction in
the usual way, with each rectangle replaced by an affine copy of $E_1$, and let $E =  \cap_{n\geq1} E_n$ be the limiting set obtained. 
Let $Q \subset [0,1]^3$ be a cube that $Q \cap E_1= \emptyset$ and 
$V_Q:= [0,1]^3\cap \pi_{xy}^{-1}(\pi_{xy}(Q))$ where $\pi_{xy}$ is the orthogonal projection from $\R^3$  to plane $\R_{xy}$, here $\R_{xy}= \{(x,y, z): z=0\}$. 

\begin{pro}
Let $E$ be a $BM$ sponge, then $E$ is thin for doubling measures.
\end{pro}
\begin{proof}
Let $R$ be a $n$-th rectangle of $E_n$. There is $n(R) \in \N$ such that 
\[
u^{-n}\leq p^{-n-n(R)}<u^{-n+1}.
\]
Divide $R$ (in the same way as the construction of $E_1$) $n(R)$ times into ($p \times q \times u)^{n(R)}$ rectangles. Let 
\[
I(R):= \{(i,j,k): 1\leq i \leq p^{n(R)} , 1\leq j \leq q^{n(R)}, 1\leq k \leq u^{n(R)}\}. 
\]
We may write $R= \cup_{\sigma \in I(R)} R_\sigma$. For $\sigma=(i,j,k)$ let $f_{i,j,k}:=f_\sigma$. Denote $R_{i,j}= \cup^{u^{n(R)}}_{k=1} R_{i,j,k}$ and $R_i= \cup^{q^{n(R)}}_{j=1} R_{i,j}$. For each $R_\sigma, \sigma \in I(R)$, there is an affine map $f_\sigma$ such that $R_\sigma= f_\sigma([0,1]^3)$. 

Denote $G(R)= \cup_{\sigma \in I(R)}f_\sigma(Q)$. Let $\mu \in \D([0,1]^3)$. We are going to prove
$\mu(G(R)) \gtrsim \mu(R)$
where $\gtrsim$ means there is a constant $C$  depends on $\mu$ only such that $\mu(G(R)) \geq C\mu(R)$. For the convenience in what follows we will use notation $\gtrsim$ when there is a constant depending on $\mu$ only. Note that the constant  may be different in different places. Applying the similar argument as in Lemma  \ref{dd}, it's not hard to see 
\begin{equation}
\mu(f_{\sigma}(Q))\gtrsim \mu(f_{\sigma}(V_Q)), \sigma \in I(R).
\end{equation}
Thus
\begin{equation}\label{h1}
\sum^{u^{n(R)}}_{k=1} \mu( f_{i,j,k}(Q))
\gtrsim \sum^{u^{n(R)}}_{k=1}\mu(f_{i,j,k} (V_Q))= \mu(V_Q(i,j)),
\end{equation}
where $V_Q(i,j):= \cup^{u^{n(R)}}_{k=1} f_{i,j,k}(V_Q)$. Let $V_Q'(i,j):=R_{i,j}\cap  \pi_{zx}^{-1}(\pi_{zx}V_Q(i,j)).$
Applying the same argument as in Lemma  \ref{dd} again, we get
\begin{equation}\label{h2}
\mu(V_Q(i,j))\gtrsim \mu(V'_Q(i,j))
\end{equation}
and 
\begin{equation}\label{h3}
\mu(\cup^{j=q^{n(R)}}_{j=1} V'_Q(i,j))\gtrsim \mu(R_i).
\end{equation}
Note that $f_\sigma(Q)$ are pair disjoint for $\sigma \in I(R)$. Combine the estimates \eqref{h1}, \eqref{h2} and \eqref{h3} we arrive 
\begin{equation}\label{h4}
\mu(G(R)) \gtrsim \mu(R).
\end{equation}

Let $G_n := \cup_{R \in E_n} G(R)$,
then by estimate \eqref{h4}, we have  that 
\begin{equation}
\mu(G_n)\gtrsim \mu(E_n) \geq \mu(E).
\end{equation}
Apply the same argument as in the proof of Theorem \ref{carpets}, we obtain the result.
\end{proof}

\begin{rem}
It can be believed that high dimensional Bedford-McMullen self-affine sets 
are thin for doubling measures. 
Note that Bedford-McMullen self-affine sets in $\R^d, d\geq 2$ satisfies $OSCH$. 
\end{rem}

\section{Purely atomic measures}
We say that a  measure is \emph{purely  atomic}, if it has full measure on a countable set. For the results related purely atomic doubling measures, see \cite{cs,kw, lww, lw, www2}. It was asked in \cite{lww} whether there exist compact set $X \subset \R$ with positive Lebesgue measure so that all doubling measures $\mu$ on $X$ are purely atomic. The answer is negative given by \cite{cs,lw}. In \cite{cs,lw}, they proved that any compact set of $\R^d$ with positive Lebesgue measure carries a doubling measure which is not purely atomic. We extend their result in the following way. 

\begin{pro}\label{d-homo}
Let $X$ be a closed subset of $\R^{d}$ with positive Lebesgue measure, then every $d$-homogeneous measure on $X$ is not purely atomic; furthermore, let $E\subset X$ and $\L^d(E)>0$, then $\mu(E)>0$ for every $d$-homogeneous measure $\mu$ on $X$. 
\end{pro}
A measure $\mu$ is called an \emph{$s$-homogeneous measure} on $X$ if there is a constant $C$ such that  for any $\lambda \geq 1$,
\[
0<\mu(B(x,\lambda r)) \leq C \lambda^{s}\mu(B(x,r))<\infty.
\]
Denote by $\mathcal{D}_{s}(X)$ all  $s$-homogeneous measure on $X$. It's easy to see that $\mathcal{D}(X)= \bigcup_{s>0}\mathcal{D}_{s}(X).$ The $s$-homogeneous measures are related to $s$-homogeneous spaces,  see \cite{l, vk}. 

\begin{proof}[Proof of Proposition \ref{d-homo}.]
Let $E\subset X$ and $\L^d(E)>0$. Let  $\mu \in\mathcal{D}_{d}(X)$ (in \cite{ls} they proved that 
$\mathcal{D}_{d}(X)\neq \varnothing$).  We are going to prove that $\mu(E)>0$ (this implies that $\mu$ is not purely atomic).

We consider $B(x,r)$ as an open ball of metric space $X$ (induced metric from $\R^d$) in the following for the convenience. Let $x_0 \in X$, then there exists $n_0$ such that $\L^d(E\cap B(x_0, n_0)) >0$. Since $\mu \in \mathcal{D}_{d}(X)$, there is constant $C$ such that for any ball $B(x,r) \subset  B(x, n_0)$, we have
\begin{equation}\label{dhome}
\mu(B(x,n_0) \leq C (\frac{n_0}{r})^d \mu(B(x,r)). 
\end{equation}
Applying the doubling property of $\mu$, we have that there is a constant $C_1$ such that  
$\mu( B(x_0, n_0))\leq C_1 \mu(B(x, n_0))$ for any $x \in B(x_0,n_0)$. Thus there is a positive constant $C_2$ such that
\begin{equation}\label{lateruse}
 \mu (B(x,r))\geq  C_2 r^d \text{ for any } B(x,r) \subset B(x_0, n_0).
\end{equation}
It's well known that  \eqref{lateruse} implies (see \cite[p.95]{ma})
\begin{equation}\label{end}
\mu(A)\geq C_3 \L^d(A) \text{ for any } A \subset B(x_0, n_0),
\end{equation}
where $C_3$ is positive constant depends on $\mu,d, n_0$ only. By the monotone property of $\mu$ and estimate \eqref{end}, we have
\[
\mu(E) \geq \mu(E \cap B(x_0, n_0))\geq C_3 \L^d(E \cap B(x_0, n_0))> 0. 
\] 
We complete the proof by the arbitrary choice of $\mu \in \D_d(X)$.
\end{proof}

\begin{comment}
For any $\epsilon>0$, there exists open set $U$ such that $E_0 \subset U\subset B( x_0, 2n_0)$ and 
\begin{equation}\label{regular}
\mu(U) \leq \mu(E_0) +\epsilon.
\end{equation}

Denote $\F:= \{B(x,r) :  B(x,r) \subset U, r > 0\}.$ Applying Vitali's  covering theorem \cite[Chapter 1]{ma}, there exists a countable of disjoint balls $\{ B_i\}_{i\geq 1}$  of $\F$ such that 
\[
\bigcup_{B \in \F} B \subset \bigcup_{i\geq 1} 5B_i.
\]
Since $E_0\subset U \subset\{5B_i\}_{i\geq1}$, thus
\begin{equation}\label{n}
\L^d(E_0) \leq \sum_{i\geq 1}\L^d(5B(x_i,r_i) )\leq \sum_{i\geq 1} 10^d r_i^{d}.
\end{equation}
Applying estimate \eqref{lateruse} and \eqref{regular}, we have  the upper bound of \eqref{n} is 
\[ 
\sum_{i\geq 1}10^d C_2^{-1} \mu(B_i) \leq 10^dC_2^{-1} \mu(U) \leq 10^d C_2^{-1} \mu(E_0)+ 10^dC_2^{-1}\epsilon,
\]
the first inequality holds by the disjoint of $\{B_i\}_{i \geq 1}$. Let $\epsilon\rightarrow 0$, then we have $\L^d(E_0) \leq 10^d C_2^{-1}\mu(E_0)$.
\end{comment}


\begin{thebibliography}{}

\bibitem{bar}  K. Bara\'nski. Hausdorff dimension of the limit sets of some planar geometric constructions.
Adv. Math., 210(1):215-245, 2007.

\bibitem{b} T. Bedford. Crinkly curves, Markov partitions and dimension. PhD thesis, University of Warwick, 1984.




\bibitem{cs}M. Cs\"ornyei, V. Suomala, On Cantor sets and doubling measures. J. Math. Anal. Appl. 393 (2012), no. 2, 680-691.

\bibitem{eg}L. C. Evans, R. F. Gariepy, Measure theory and fine properties of functions, Studies in Advanced Mathematics,CRC PRESS. 1993.


\bibitem{f} K. Falconer, Fractal Geometry: Mathematical Foundation and Applications, John Wiley, 1990.


\bibitem{Fraser} Jonathan M. Fraser,
Assouad type dimensions and homogeneity of fractals,
Trans. Amer. Math. Soc.  366  (2014),  no. 12, 6687-6733. 


\bibitem{g} J. Garnett, R. Killip, R. Schul, A doubling measure on  $\mathbb{R}^{d}$
can charge a rectifiable curve, Proc. Amer. Math. Soc. 138(2010),
1673-1679.

\bibitem{h} J. Heinonen. Lectures on analysis on metric spaces. Springer-Verlag, New York,
2001.

\bibitem{hww} D. Han, L. Wang, and S. Wen. Thickness and thinness of
uniform Cantor sets for doubling measures. Nonlinearity, 22(2009),
545-551.

\bibitem{krs} A. K\"aenm\"aki, T. Rajala, and V. Suomala: Existence of doubling measures via generalized
nested cubes. Proc. Amer. Math. Soc. 140, 2012, 3275-3281.


\bibitem{kw} R. Kaufman, J. M. Wu, Two problems on doubling measures,
Rev. Math. Iberoamericana, 11(1995),527-545.


\bibitem{kmw} L. V. Kovalev, D. Maldonado, and J.-M. Wu, Doubling measures,
monotonicity, and quasiconformality, Math. Z. 257(2007), 525-545.

\bibitem{ls} J. Luukkainen and E. Saksman. Every complete doubling metric space carries a doubling measure. Proc. Amer. Math. Soc., 126
(1998)531-534.


\bibitem{l} J. Luukkainen, Assouad dimension: antifractal metrization, porous
sets, and homogeneous measures, J. Korean Math. Soc. 35(1998),
23-76.

\bibitem{lww} M.L. Lou, S.Y. Wen, M. Wu, Two examples on atomic doubling measures, J. Math. Anal. Appl. 333 (2007) 1111-1118.

\bibitem{lw} M. Lou, M. Wu, Doubling measures with doubling continuous part,
Proc. Amer. Math. Soc. 138(2010), 3585-3589.


\bibitem{luu} J. Luukkainen, Assouad dimension: antifractal metrization, porous sets, and homogeneous measures. J. Korean Math. Soc.  35  (1998),  no. 1, 23-76.

\bibitem{Mackay} John M. Mackay, Assouad dimension of self-affine carpets, Conform. Geom. Dyn. 15 (2011), 177-187.

\bibitem{ma} P. Mattila, Geometry of sets and measures in Euclidean spaces,
Cambridge University Press, Cambridge, 1995.


\bibitem{mc} C. McMullen. The Hausdorff dimension of general Sierpi\'{n}ski carpets. Nagoya
Math. J., 96(1984), 1-9.




\bibitem{ojala} T. Ojala, On ($\alpha_n$)-regular sets, Ann.Acad. Sci. Fenn. Math., 39:655-673,2014.

\bibitem{tt} T. Ojala, T. Rajala, A function whose graph has positive doubling measure, arxiv.org/abs/1406.4693

\bibitem{ors} T. Ojala, T. Rajala, V. Suomala, Thin and fat sets for doubling measures in metric spaces, 
Studia Math. 208 (2012), 195-211.

\bibitem{pw} F.J. Peng and S.Y. Wen. Fatness and thinness of uniform cantor sets for doubling measures. Sci. China Math., 54:75–81, 2011.

\bibitem{s} E. M. Stein, Harmonic analysis: real-variable methods,
orthogonality, and oscillatory integrals. Princeton Mathematical
Series, 43. Princeton University Press, Princeton, NJ, 1993.

\bibitem{vk}A. L. Vol'berg and S. V. Konyagin. On measures with the doubling condition. Izv. Akad.
Nauk SSSR Ser. Mat., 51(3):666-675, 1987. MR903629 (88i:28006)


\bibitem{www1}W. Wang, S. Wen and Z. Wen, Fat and thin sets for doubling measures in Euclidean space, Ann. Acad. Sci. Fenn. Math., 38:535-546, 2013.



\bibitem{www2} W. Wang, S. Wen, Z. Wen, Note on atomic doubling measures, quasisymmetrically thin sets and thick sets, J. Math. Anal. Appl. 385 (2012), 1027-1032.


\bibitem{wu1} J.-M. Wu, Null sets for doubling and dyadic doubling
measures, Ann. Acad. Sci. Fenn. Ser. A I Math. 18 (1993), 77-91.



\bibitem{wu2} J.-M. Wu, Hausdorff dimension and doubling measures on metric spaces. Proc.
Amer. Math. Soc. 126 (1998), 1453-1459.


\end{thebibliography}
\end{document}